\documentclass[12pt,reqno]{amsart}

\usepackage{latexsym}
\usepackage{amssymb}
\usepackage[matrix,arrow]{xy}

\author{Florin Ambro} 
\address{ 
Institute of Mathematics ``Simion Stoilow'' of the Romanian Academy\\
P.O. BOX 1-764, RO-014700 Bucharest\\ 
Romania.
        }
\email{florin.ambro@imar.ro}

\newcommand{\isoto}{{\overset{\sim}{\rightarrow}}}

\newcommand{\Q}{{\mathbb Q}}
\newcommand{\Z}{{\mathbb Z}}
\newcommand{\N}{{\mathbb N}}
\newcommand{\R}{{\mathbb R}}

\newcommand{\bA}{{\mathbb A}} 
\newcommand{\bC}{{\mathbb C}} 

\newcommand{\fm}{{\mathfrak m}} 

\newcommand{\cB}{{\mathcal B}} 

\newcommand{\emb}{\operatorname{emb}}

\newcommand{\GL}{\operatorname{GL}}

\newcommand{\Hom}{\operatorname{Hom}}

\newcommand{\ind}{\operatorname{index}}
\newcommand{\Int}{\operatorname{int}}

\newcommand{\lct}{\operatorname{lct}}

\newcommand{\mld}{\operatorname{mld}}

\newcommand{\mult}{\operatorname{mult}}

\theoremstyle{plain}
\newtheorem{thm}{Theorem}[section]

\newtheorem{lem}[thm]{Lemma}
\newtheorem{cor}[thm]{Corollary}

\newtheorem{prop}[thm]{Proposition}

\theoremstyle{definition}
\newtheorem{defn}[thm]{Definition}

\newtheorem{exmp}[thm]{Example}

\newtheorem{rem}[thm]{Remark}

\newtheorem{ack}{Acknowledgments}   

\theoremstyle{remark}

\begin{document}

\bibliographystyle{amsalpha+}
\title[Classification]
{Classification of toric surface singularities}
\maketitle

\dedicatory{
\center{Dedicated to Yuri Prokhorov, on the occasion of his sixtieth birthday}

}

\begin{abstract} 
We classify two-dimensional toric log germs in terms of their minimal log discrepancy. 
In particular, Borisov's series of toric surface singularities with minimal log discrepancy 
bounded below can be described explicitely and geometrically.
\end{abstract}






\section{Introduction}


Terminal, canonical, log terminal, log canonical singularities are classes of singularities that inevitably appear while running the Log Minimal Model Program. Such singularities are very special in the larger class of normal singularities, so a first hope would be to classify them {\em explicitely}. In low dimension, this is possible.

Consider the surface case (see~\cite{Mat02} for example). Terminal just means smooth. Canonical singularities are the Du Val singularities, which up to analytic isomorphism have explicit equations of type A-D-E.
Log terminal singularities are analytically isomorphic to quotients $0\in \bC^2/G$, where $G\subset \GL(2,\bC)$ is a finite group acting without reflections. Log canonical surface singularities are classified either by taking the index one cover and descending information from the Gorenstein case, or via the dual graph of exceptional curves on the minimal resolution.

In dimension three (see~\cite{Reid87} for example) less is known, besides 
a partial classification of canonical singularities and the classification of terminal $3$-fold singularities (described as cyclic covers of certain hypersurface singularities). As a consequence of classification, M. Reid~\cite{Reid87} observed that for a terminal $3$-fold singularity $P\in X$, the general anti-canonical member $S\in |-K_X|$ has at most Du Val singularity at $P$. This observation had two crucial consequences.
First, an analogue for the general anti-bicanonical member of a flipping contraction was the final step in establishing the Mori Program in dimension three~\cite{Mori88}. Second, V.V. Shokurov~\cite{Sho92} introduced the theory of complements and used them to establish the existence of $3$-fold log flips.

Since~\cite{Sho92,Sho04}, it is believed that in higher dimension a {\em qualitative} classification of 
singularities may be possible, according to certain properties of the anti-pluricanonical members and the minimal log discrepancy (see~\cite{Am06} for an introduction to minimal log discrepancies and Shokurov's conjectures).

Toric singularities are combinatorially defined, so it is natural to know more on their classification. Toric terminal $\Q$-factorial singularities are in bijection with lattice simplices, with one vertex at the origin, containing
no other lattice points besides vertices. 
In dimension three, these are the cyclic quotient singularities of type $\frac{1}{r}(1,-1,a)$, where $r,a$ are relatively prime positive integers. In dimension four, a complete classification was only recently completed~\cite{IS21}, after initial computer-generated explicit partial lists of series proposed in~\cite{MMS88}. Renouncing the explicit nature of the partial lists, A. Borisov~\cite{Bor99} obtained a qualitative classification of toric singularities with minimal log discrepancy bounded below:
given $d,\epsilon>0$, there exists only finitely many {\em series} of toric 
$\Q$-factorial singularities with $\dim X=d$ and $\mld(X)\ge \epsilon$. The series are defined by finitely many closed subgroups of the real torus $\R^d/\Z^d$.
Our motivation is to understand Borisov's series, and to see if their combinatorial definition has a geometric equivalent, possibly with an analogue in the non-toric case. 
In this note, we completely describe Borisov's series in dimension two. An interesting new feature is that they are defined by invariant hyperplane sections.

Our setup differs slightly from Borisov's. For a germ $P\in X$, Borisov considers the {\em global} minimal log discrepancy condition 
$\mld(X)\ge \epsilon$, while we consider the {\em local} minimal log discrepancy condition $\mld_P(X)\ge t$. That is, Borisov considers all geometric valuations over $X$ whose center contains $P$, while we only consider those with center exactly $P$. Borisov's global series are just glueings of finitely many local series.

Fix a real number $0< t\le 2$. Suppose for simplicity the boundary is zero. Every toric surface singularity $P\in X$ is cyclic quotient, say of type $x=\frac{1}{r}(w_1,w_2)$. The complement $X\setminus T$ of the torus sitting inside $X$ is an invariant boundary $\Sigma=E_1+E_2$ satisfying $K+\Sigma=0$ and $\mld_P(X,\Sigma)=0$. The condition $\mld_P(X)\ge t$ translates as follows: for every $y\in (0,1]^2\cap (\Z^2+\Z x)$, the inequality $y_1+y_2\ge t$ holds. A C-program can list all $x$ with this property, just like in~\cite{MMS88}. There are infinitely many solutions. We take each solution, and investigate why it sits in the list. Our first observation is that besides finitely many exceptions, $x$ sits in the list for a very simple reason: $m_1x_1+m_2x_2\in \Z$ for some integers $m_1,m_2 \in [0,\frac{1}{t}]$, not both zero. We say in this case that $P\in X^2$ belongs to a {\em codimension one $t$-lc series}.
Note that there are only finitely many codimension one $t$-lc series. Geometrically, they can be characterized in two equivalent (dual) ways:
\begin{itemize}
	\item[(H)] The principal divisor $H=m_1E_1+m_2E_2$ is an invariant hyperplane section of $P\in X$ satisfying $\mld_P(X,tH)\ge 0$.
	\item[(C)] Let $n=\max(m_1,m_2)$ and $B_n=\Sigma-\frac{H}{n}=\sum_{i=1}^2(1-\frac{m_i}{n})E_i\ge 0$. Then
	$n(K+B_n)\sim 0$ and $\mld_P(X,B_n)\ge \frac{1}{n}\ge t$. Here $B_n$ is an invariant $n$-complement of $K$.
\end{itemize}

Our second observation is that even the finitely many exceptions can be described explicitely, using two "bounded, linearly independent" invariant hyperplanes (see Theorem~\ref{Lawr} for the precise statement).  Thus toric surface singularities $P\in X$ with $\mld_P(X)\ge t$ belong either to (finitely many) codimension one $t$-lc series, or to finitely many isomorphism types.

Our third observation is that the above classification extends to the case when $X$ is endowed with an invariant boundary $0\le B\le \Sigma$ (see Theorem~\ref{2se} and Section 5.2)). In the log case, the codimension one $t$-lc series are (dually) characterized as follows:
 \begin{itemize}
 	\item[(H)] The principal divisor $H=m_1E_1+m_2E_2$ is an invariant hyperplane section of $P\in X$ satisfying $\mld_P(X,B+tH)\ge 0$.
 	\item[(C)] Let $B=\sum_{i=1}^2(1-a_i)E_i$. Let $n\in \Z_{\ge 1}$ be minimal such that 
 	$na_i\ge m_i$ for $i=1,2$. Denote $B_n=\Sigma-\frac{H}{n}=\sum_{i=1}^2(1-\frac{m_i}{n})E_i\ge B$. 
 	Then $n(K+B_n)\sim 0$, $\mld_P(X,B_n)\ge \frac{1}{n}$, and either a) $n\le t^{-1}$, or 
 	b) $t^{-1}<n\le \lceil t^{-1}\rceil$ and $B_n\ge (1-\frac{1}{nt})\Sigma+\frac{1}{nt}B$. 
 \end{itemize}
The reader may check that the two equivalent conditions imply $\mld_P(X,B)\ge t$. 

Our main result (Theorem~\ref{2se}) is in fact a characterization 
of the condition $\mld_P(X,B)\ge t$ in terms of invariant hyperplanes. 
But once the type of coefficients of $B$ are fixed, Theorem~\ref{2se}
easily gives explicit forms for the series. 

We outline the structure of this note. Section 2 is a preliminary on subgroups avoiding open cones inside a vector space. Section 3 is our key technical tool, a classification of subgroups of $\R^2$ which do not intersect a deformation of the standard simplex (compare~\cite{Law91}).
In Section 4 we apply the classification in Section 3 to translate into
hyperplane sections the condition on a singularity to be $t$-log canonical. This provides a proof for the lists generated by a C-program. In Section 5, we interpret the classification in terms of hyperplane sections and anti-pluricanonical members.

\begin{ack} 
	My interest in Borisov's series was sparked by very stimulating discussions with Yuri Prokhorov, in RIMS about 20 years ago, on what was known at the time as McKernan-Shokurov conjecture on the singularities of Fano fibrations. 
	
	I would like to thank V.V. Shokurov and the anonymous referee for useful comments and remarks.
\end{ack}


\section{Preliminary}


Let $V$ be a finite dimensional $\R$-vector space.
Let $V^*=\Hom_\R(V,\R)$ be the dual vector space. The duality pairing 
$$
V^*\times V\to \R, \ (\varphi,v)\mapsto \langle \varphi,v\rangle=\varphi(v)
$$ induces a canonical identification
$V\isoto (V^*)^*$. For a convex cone $\sigma\subset V$, the dual cone 
is 
$$
\sigma^\vee=\cap_{v\in \sigma}\{\varphi\in V^* \vert \langle \varphi,v\rangle\ge 0\}.
$$ 
The duality theorem for convex cones states that 
the closure of $\sigma$ in $V$ equals $(\sigma^\vee)^\vee$.
Let $G\le (V,+)$ be a subgroup. The dual subgroup inside $V^*$ is defined by 
$$
G^*=\{\varphi\in V^* \vert \langle \varphi,G\rangle\subset \Z\}.
$$
The duality theorem for subgroups states that the closure of $G$ in $V$ equals $(G^*)^*$.

If we choose a basis for $V$, and the dual basis for $V^*$, we may identify $V=\R^d$, $V^*=\check{\R}^d$, with duality pairing 
$\langle x,\check{x}\rangle=\sum_{i=1}^dx_i\check{x}_i$.

For $v\in V$, denote $v^*=\{\varphi\in V^*\vert \langle \varphi,v\rangle\in \Z\}$, that is the subgroup dual to the subgroup $\Z v\subset V$.

\begin{lem}\label{co}
	Let $V$ be a finite dimensional vector space, let $G\le V$ be a closed subgroup, let $\sigma\subset V$ be a closed convex cone such that $\dim \sigma=\dim V$. The following are equivalent:
	\begin{itemize}
		\item[(1)] $G\cap \Int \sigma=\emptyset$. 
		\item[(2)] There exists $0\ne \varphi\in \sigma^\vee$ such that
		$\R \varphi\subset G^*$.
	\end{itemize}
\end{lem}

\begin{proof} We may restate (2) as $\varphi\in G^\perp\cap \sigma^\vee$. Therefore (2) is equivalent to $G^\perp\cap\sigma^\vee\ne 0$.
We have $(G^\perp)^\perp=G\otimes_\Z \R=:G_\R$. By the separation theorem for convex sets, $G^\perp\cap\sigma^\vee= 0$ if and only if 
$G_\R+ \sigma=V$. Since $\sigma$ has an interior point, the latter is equivalent to $G_\R\cap \Int\sigma\ne 0$.
We conclude that (2) is equivalent to $G_\R\cap \Int\sigma=0$.
It remains to show that $G\cap \Int\sigma\ne \emptyset$ if and only if 
$G_\R\cap \Int\sigma\ne \emptyset$. 

From the inclusion $G\subset G_\R$,
the direct implication is clear. For the converse, let $g_i\in G$ and $t_i\in \R$ with $\sum_it_ig_i\in \Int \sigma$. There exists $\epsilon>0$
such that $\sum_i(t_i-\epsilon,t_i+\epsilon)g_i\subset \Int\sigma$.
Then $\sum_i(tt_i-t\epsilon,tt_i+t\epsilon)g_i\subset \Int\sigma$ for $t>0$.
For $2t\epsilon> 1$, there exist integers $z_i\in (tt_i-t\epsilon,tt_i+t\epsilon)$. Then
$\sum_iz_ig_i\in \Int\sigma\cap G$.
\end{proof}

\begin{lem}\label{cc} Let $V$ be a finite dimensional vector space, let $G\le V$ be a closed subgroup, let $\sigma\subset V$ be a closed convex cone such that $\dim \sigma=\dim V$. If $\dim V\in \{1,2\}$, the following are equivalent:
\begin{itemize}
\item[(1)] $G\cap \sigma=0$. 
\item[(2)] There exists $\varphi\in \Int\sigma^\vee$ such that $\R \varphi\subset G^*$.
\end{itemize}
\end{lem}

\begin{proof} We may restate (2) as $\varphi\in G^\perp\cap \Int \sigma^\vee$. By the separation theorem for convex sets,
	$G^\perp\cap \Int \sigma^\vee =\emptyset$ if and only if 
	$G_\R\cap \sigma\ne 0$. Therefore (2) is equivalent to 
	$G_\R\cap \sigma= 0$. It remains to show that 
	$G\cap \sigma=0$ if and only if 
	$G_\R\cap \sigma=0$. 
	
	Suppose $G\cap \sigma=0$. Then $G\cap \Int\sigma=\emptyset$.
	By Lemma~\ref{cc}, there exists $0\ne \varphi\in \sigma^\vee$ 
	such that $G\subset \varphi^\perp$. If $\dim V=1$, we obtain $G=0$.
	Then $G_\R=0$, hence $G_\R\cap \sigma=0$. Suppose $\dim V=2$. We may 
	choose an isomorphism $V\simeq \R^2$ such that $\sigma$ corresponds to 
	$\R^2_{\ge 0}$ or $\R_{\ge 0}\times \R$. In the latter case $G=0$, and we argue as above. Suppose $\sigma=\R^2_{\ge 0}$. If $\varphi\in \Int\sigma^\vee$, then $G_\R\cap \sigma\subset \sigma\cap \varphi^\perp=0$. Else, we may suppose $\varphi=(0,1)$. Then 
	$G_\R\subset \R(1,0)$. Then $G\cap \R_{\ge 0}=0$. Therefore $G=0$.
\end{proof}

\begin{exmp} Lemma~\ref{cc} is no longer true if $\dim V\ge 3$.
	For example, let $c\in \R\setminus \Q$, $V=\R^3$, $G=\Z(0,1,-1)+\Z(1,c,-c)$, $\sigma=\R^3_{\ge 0}$. Then $G\cap \sigma=0$,  $G_\R=\{x\in \R^3\vert x_2+x_3=0\}$, $G_\R\cap \sigma=\R_{\ge 0}(1,0,0)$, $G^\perp=\R(0,1,1)$.
\end{exmp}

\begin{lem}\label{d1}
	Let $G\le \R$ be a subgroup with dual subgroup $G^*\le \check{\R}$,
	let $t>0$. The following properties are equivalent:
	\begin{itemize}
		\item[(1)] $G\cap (0,t)=\emptyset$.
		\item[(2)] $G^*\cap (0,\frac{1}{t}]\ne \emptyset$.
	\end{itemize}
\end{lem}

\begin{proof} 
	Suppose $y\in G^*\cap (0,\frac{1}{t}]$. Let $x\in G$ with $x>0$.
	Then $0<xy\in \Z$. Therefore $xy\ge 1$. Therefore $x\ge t$. We deduce $G\cap (0,t)=\emptyset$.
	
	Conversely, suppose $G\cap (0,t)=\emptyset$. We may suppose $G$ is closed. Then either $G=0$, or $G=\Z g$ for some $g\ge t$. If $G=0$,
	then $G^*=\check{\R}$ contains $(0,\frac{1}{t}]$. If $G=\Z g$, then 
	$G^*=\Z \frac{1}{g}$, and $0<\frac{1}{g}\le \frac{1}{t}$.	
\end{proof}

\begin{lem}\label{d0} Let $G\le \R^2$
be a subgroup such that $G\cap (0,+\infty)^2\ne \emptyset$. Let $t>0$ and $U=(0,t)\times (0,+\infty)$.
The following are equivalent:
\begin{itemize}
\item[(1)] $G\cap U=\emptyset$.
\item[(2)] There exists $(\check{x}_1,0)\in G^*$ with $0<\check{x}_1\le\frac{1}{t}$.
\end{itemize}
\end{lem}

\begin{proof} Define $\tilde{U}=(0,t)\times \R \subset \R^2$. We claim that 
$G\cap U=\emptyset$ if and only if $G\cap \tilde{U}=\emptyset$. Indeed,
suppose by contradiction that $G\cap U=\emptyset\ne G\cap \tilde{U}$. There exists $g=(x',y')\in G$ with $0<x'<t$ and $y'\le 0$. We may enlarge $t$ and suppose $t=\inf\{x \vert \exists \ (x,y)\in G\cap (0,+\infty)^2\}$.
If $(x,y)\in G\cap (0,+\infty)^2$ then
$(x-x',y-y')\in G\cap (0,+\infty)^2$. At the limit, we obtain 
$t\le t-x'$, that is $x'\le 0$. Contradiction!

Consider the projection $\pi\colon \R^2\to \R$, $\pi(x,y)=x$.
Then $\pi(G)=H$ is a subgroup of $\R$ and $\tilde{U}=\pi^{-1}((0,t))$.
Therefore (1) is equivalent to
$H\cap (0,t)=\emptyset$. By Lemma~\ref{d1}, this is equivalent to
$H^*\cap (0,\frac{1}{t}]\ne \emptyset$. The latter is equivalent to (2).
\end{proof}

\begin{lem}\label{ca}
Let $G\le V$ and $\varphi\in V^*\setminus 0$. The following are equivalent:
\begin{itemize}
\item[(1)] $G\le \varphi^*$, $G\cap \varphi^\perp=0$.
\item[(2)] $G=0$, or $G=\Z v$ with $\langle \varphi,v\rangle\in \Z_{>0}$.
\end{itemize}
\end{lem}

\begin{proof} The implication $(2)\Longrightarrow (1)$ is clear.
	For the converse, consider the subgroup $\langle \varphi,G\rangle\le \Z$. If zero, then $G=G\cap \varphi^\perp=0$. If non-zero, it equals $q\Z$ for some $q\ge 1$. Choose $v\in G$ with $\varphi(v)=q$. Then 
	$G=\Z v+G\cap \varphi^\perp=\Z v$.
\end{proof}

\begin{lem}\label{cd}
Let $G\le V$ and $\varphi \in V^*\setminus 0$. The following
are equivalent:
\begin{itemize}
\item[(1)] $\langle \varphi,G\rangle=\Z$ and $G\cap \varphi^\perp$ is dense in $\varphi^\perp$.
\item[(2)] $G$ is a dense subgroup of $\varphi^*$.
\end{itemize}
\end{lem}

\begin{proof} Any subgroup $G\le \varphi^*$ is either a 
	subgroup of $\varphi^\perp$, or of the form $G=\Z v+G\cap \varphi^\perp$ where $\varphi(v)=q\in \Z_{\ge 1}$ (and then $\varphi^*=\Z\frac{v}{q}+\varphi^\perp$). If $G$ is dense in $\varphi^*$, then $\varphi(G)$ is dense in $\Z$, hence 
	$\varphi(G)=\Z$. The equivalence follows.
\end{proof}


\section{Combinatorial classification}


Let $V$ be a $2$-dimensional $\R$-vector space. Let $\sigma\subset V$ be a closed strongly convex $2$-dimensional cone. Let $\sigma^\vee\subset V^*$ be the dual cone inside the dual vector space. Let $\psi\in \sigma^\vee\setminus 0$. Define a functional
$$
\sigma^\vee\ni m\mapsto \gamma(m):=\sup\{t\ge 0 \vert \psi-tm\in \sigma^\vee\}
\in [0,+\infty].
$$
Since $0$ is the smallest face of $\sigma^\vee$, $\gamma(0)=+\infty$ and
$\gamma(m)<+\infty$ is a maximum for every $m\in \sigma^\vee\setminus 0$.
Note $\gamma(m)=0$ if and only if $m$ does not belong to the smallest face of $\sigma^\vee$ which contains $\psi$. 
The reciprocal $\sigma^\vee\ni m\mapsto \frac{1}{\gamma(m)}\in [0,+\infty]$ is convex, positively homogeneous, piecewise linear with respect to the barycentric subdivison of $\psi\in \sigma^\vee$.
The set 
$
\{m\in \sigma^\vee \vert \gamma(m)\ge 1 \}=\sigma^\vee\cap (\psi-\sigma^\vee)
$
is a compact convex polytope.

Let $G\le (V,+)$ be a subgroup such that $G\cap \Int \sigma \ne \emptyset$.
Define
$$
\lambda:=\lambda(\sigma,\psi,G)=\inf\{\langle \psi,e\rangle \vert  e\in G\cap \Int \sigma \}
$$
and suppose $\lambda>0$.

\begin{lem}\label{ex} 
There exists $m\in G^*\cap \sigma^\vee \setminus 0$
such that $\gamma(m)>0$.
\end{lem}

\begin{proof} We may suppose $V=\R^2$, $\sigma=\R_{\ge 0}^2$ and $\psi=(a_1,a_2)\in \check{\R}^2$ with 
$a_1,a_2\ge 0$. 

Case $\psi\in \partial(\sigma^\vee)$. Here $\psi=(a_1,0)$
and $a_1>0$. Then $G\cap (0,\frac{\lambda}{a_1})\times (0,\infty)=\emptyset$.
By Lemma~\ref{d0}, there exists $(m_1,0)\in G^*$ with 
$0<m_1\le\frac{a_1}{\lambda}$. Then $(m_1,0)\in G^*\cap \sigma^\vee\setminus 0$
and $\gamma(m_1,0)=\frac{a_1}{m_1}\ge \lambda>0$.

Case $\psi\in \Int \sigma^\vee$. Here $a_1,a_2>0$.
If $G^*\cap \sigma^\vee= \{0\}$,
Lemma~\ref{cc} implies that there exists $e\in \Int \sigma$ with 
$\R e\subseteq \overline{G}$. Then $\lambda=0$, a contradiction.
Therefore there exists $m\in G^*\cap \sigma^\vee\setminus 0$.
In dual coordinates, $m=(m_1,m_2)$ and $\gamma(m)=\min(\frac{a_1}{m_1},\frac{a_2}{m_2})>0$.
\end{proof}

\begin{defn}
$\gamma:=\gamma(\sigma^\vee,\psi,G^*)=\sup\{\gamma(m) \vert m\in G^*\cap \sigma^\vee\setminus 0\}$.
\end{defn}

In particular, $\gamma>0$. We abused notation by denoting by $\gamma$ both a functional and a number,
but we hope the reader will easily distinguish them in context. 

\begin{lem}
$\lambda\ge \gamma$.
\end{lem}

\begin{proof}
Let $m\in G^*\cap \sigma^\vee\setminus 0$. Then $\psi-\gamma(m)m\in \sigma^\vee$.
Let $e\in G\cap \Int \sigma$. Then $\langle m,e\rangle$ is a positive integer, so $\langle m,e\rangle\ge 1$. Therefore 
$\langle \psi,e\rangle\ge \langle\gamma(m)m,e\rangle\ge \gamma(m)$. We 
obtain
$$
\langle \psi,e\rangle\ge \gamma(m).
$$
Taking infimum after all $e$, and supremum after all $m$, we obtain
the claim.
\end{proof}

\begin{lem}\label{gmax}
$\gamma$ is a maximum.
\end{lem}

\begin{proof} Since $\gamma\le \lambda$, $\gamma$ is a 
positive real number. Let $c>0$ with $c\gamma>1$.
Let $m^n\in G^*\cap \sigma^\vee\setminus 0$ be a sequence with
$\gamma(m^n)\uparrow \gamma$.
Then $c\gamma(m^n)>1$ for $n\gg 0$, and therefore 
$m^n\in (c\psi-\sigma^\vee)\cap \sigma^\vee$ for $n\gg 0$. 
The latter set is compact.
Therefore the sequence $m^n$ has a cluster point $m$. It follows that 
$m\in G^*\cap \sigma^\vee$ and $\gamma(m)=\gamma$. In particular, $m\ne 0$.
\end{proof}

By Lemma~\ref{gmax}, there exists $v^1\in G^*\cap \sigma^\vee\setminus 0$ 
with $\gamma(v^1)=\gamma$. 

\begin{lem}
$\langle v^1,G\rangle= \Z$.
\end{lem}

\begin{proof}
We have $\langle v^1,G\rangle\le \Z$.
Since $G\cap \Int \sigma \ne \emptyset$, $\langle v^1,G\rangle\ne 0$.
Therefore $\langle v^1,G\rangle=k\Z$ for some positive integer $k$. 
Then $\frac{1}{k}v^1\in G^*\cap \sigma^\vee \setminus 0$. We
have $\gamma(\frac{1}{k}v^1)=k\gamma(v^1)=k\gamma$. Then $k\gamma\le \gamma$,
so $k=1$.
\end{proof}

\begin{lem}
Suppose $\psi\in \partial(\sigma^\vee)$. Then $\lambda=\gamma$,
$\psi=\gamma v^1$, and $\lambda$ is a minimum.
\end{lem}

\begin{proof} We use the notations and argument in the first part
of the proof of Lemma~\ref{ex}. We constructed 
$m=(m_1,0)\in G^*\cap \partial(\sigma^\vee)$ with $\gamma(m)\ge \lambda$.
But $\lambda\ge \gamma$, so $\gamma(m)=\lambda=\gamma$.
Equivalently, $\psi=\gamma m$. In particular, $\psi$ is a minimum
and $m=v^1$ (so $v^1$ is unique maximizer in this case).
\end{proof}

We suppose from now on that $\psi\in \Int \sigma^\vee$.
Consider the subgroup $G\cap (v^1)^\perp\le (v^1)^\perp$. 
Since $(v^1)^\perp\simeq \R$, we distinguish there cases:
\begin{itemize}
\item[(I)] $G\cap (v^1)^\perp=0$.
\item[(II)] $G\cap (v^1)^\perp$ is dense in $(v^1)^\perp$.
\item[(III)] $G\cap (v^1)^\perp$ is non-zero and discrete.
\end{itemize}

{\em Case (I)}: since $G\cap \Int\sigma \ne \emptyset$, Lemma~\ref{ca} 
gives $G=\Z e$ for some $e\in \Int \sigma$. Then:
\begin{itemize}
\item $\lambda$ is a minimum, $\lambda=\gamma, \psi=\gamma v^1$.
\item The set $\{m\in \sigma^\vee \vert \langle m,e\rangle=1\}\subset G^*$ is a closed segment with endpoints $m^1,m^2$, $\psi=\gamma(m^1)m^1+\gamma(m^2)m^2$ and $\lambda=\gamma(m^1)+\gamma(m^2)$.
\end{itemize}
So $v^1$ is unique maximizer in this case. Note that $m^1,m^2$ generate the cone $\sigma^\vee$.

\begin{proof} Since $G\cap \Int\sigma=\Z_{\ge 1}e$, we obtain
$\lambda=\langle \psi,e\rangle$.
Since $\langle \frac{\psi}{\lambda},e\rangle =1$, we obtain
$\frac{\psi}{\lambda}\in G^*$. We compute 
$\gamma(\frac{\psi}{\lambda})=\lambda$.
Let $m\in G^*\cap \sigma^\vee\setminus 0$. Then 
$\langle m,e\rangle=q$ is a positive integer and
$
\langle \psi-\gamma(m)m,e\rangle\ge 0.
$
Then $\gamma(m)\le \frac{\lambda}{q}$. Therefore
$\gamma(m)\le \lambda$. In particular, $\gamma=\lambda$.

We have $\langle v^1,e\rangle=1$. 
Since $\langle \psi-\gamma v^1,e\rangle=\lambda-\gamma=0$
and $e\in \Int \sigma$, we obtain $\psi=\gamma v^1$.

Choose coordinates such that $\sigma=\{(x_1,x_2)\in \R^2 \vert x_1,x_2\ge 0\}$
and $\sigma^\vee=\{(m_1,m_2)\in \check{\R}^2 \vert m_1,m_2\ge 0\}$. Let
$e=(\alpha,\beta)$ and $\psi=(a_1,a_2)$. The segment
$\{m\in \sigma^\vee \vert \langle m,e\rangle=1\}$ has endpoints $m^1=(\frac{1}{\alpha},0)$ and $m^2=(0,\frac{1}{\beta})$. One computes 
$$
\gamma(m_1,m_2)=\min(\frac{a_1}{m_1},\frac{a_2}{m_2}).
$$
Note that $a_i>0$, so if $m_i=0$ we set $\frac{a_i}{m_i}=+\infty$.
Then $\gamma(m^1)=a_1\alpha,\gamma(m^2)=a_2\beta$. In particular,
$\psi=\gamma(m^1)m^1+\gamma(m^2)m^2$. Finally,
$\lambda=\langle \psi,e\rangle=a_1\alpha+a_2\beta$.
\end{proof}

{\em Case (II)}: By Lemma~\ref{cd}, $G$ is dense in 
$(v^1)^*$. By duality for subgroups, this is equivalent to 
$G^*=\Z v^1$. Then
\begin{itemize}
\item $\lambda=\gamma$.
\item $\lambda$ is a minimum if and only if $\psi=\gamma v^1$.
\end{itemize}

\begin{proof}
By assumption, $\psi-\gamma v^1\in \partial(\sigma^\vee)$. We distinguish two cases:

Suppose $\psi-\gamma v^1=0$. Then $\lambda=\gamma$ and $\lambda$ is
a minimum.

Suppose $\psi-\gamma v^1\ne 0$. Then $\psi-\gamma v^1$ belongs to the boundary of $\sigma^\vee$. In coordinates,
$\sigma=\{(x_1,x_2) \vert x_1,x_2\ge 0\}$,
$\sigma^\vee=\{(m_1,m_2) \vert m_1,m_2\ge 0\}$, $v^1=(\alpha,\beta)$
and $\psi=(a_1,a_2)$. Then 
$$
\gamma=\min(\frac{a_1}{\alpha},\frac{a_2}{\beta}). 
$$
On the other hand, $G$ is dense in $(v^1)^*$, so 
$$
\lambda=\inf\{a_1x_1+a_2x_2 \vert x_1,x_2>0,\alpha x_1+\beta x_2\in \Z_{\ge 1} \}.
$$
It follows that $\lambda=\min(\frac{a_1}{\alpha},\frac{a_2}{\beta})$.
Therefore $\lambda=\gamma$. Finally, we check that $\lambda$ is
not a minimum. Let $e\in G\cap \Int \sigma$. 
Since $\psi-\gamma v^1\in \sigma^\vee\setminus 0$,
we have $\langle \psi-\gamma v^1,e\rangle>0$. 
Also, $\langle v^1,e\rangle\ge 1$. Then
$$
\langle \psi,e\rangle\ge \langle \psi-\gamma v^1,e\rangle+
\gamma \langle v^1,e\rangle>\gamma
$$
Therefore $\langle \psi,e\rangle>\lambda$.
\end{proof}

{\em Case (III)}: Here $G\cap (v^1)^\perp=\Z e_2$ for some $e_2\ne 0$.
After possibly replacing $e_2$ by $-e_2$, we may suppose 
$$
\psi':=\langle \gamma^{-1}\psi- v^1,e_2\rangle\ge 0.
$$
Note $\langle \psi,e_2\rangle=\gamma\psi'$. Choose $e_1\in G$ with $\langle v^1,e_1\rangle=1$. 
It follows that $G$ is a $2$-dimensional lattice with basis $e_1,e_2$.
There exist $\alpha<\beta\le +\infty$ such that 
$$
\{e\in \Int \sigma \vert \langle v^1,e\rangle=1\}=\{e_1+te_2 \vert \alpha<t<\beta\}.
$$
By assumption, $\psi-\gamma v^1\in \partial(\sigma^\vee)$. Since $\psi' \ge 0$, we obtain 
$$
\langle \psi-\gamma v^1,e_1+\alpha e_2\rangle =0.
$$
Replacing $e_1$ by $e_1+\lfloor \alpha\rfloor e_2$, we may suppose 
$$
0\le \alpha<1.
$$

Define $v^2\in G^*$ by $\langle v^2,e_1\rangle=0,\langle v^2,e_2\rangle=1$.
Then $\{v^1,v^2\}$ is a basis of the lattice $G^*$, dual to the basis 
$\{e_1,e_2\}$.

Note that $\psi'=0$ if and only if $\psi=\gamma v^1$.
Indeed, suppose $\psi'=0$. Then $\psi-\gamma v^1$ is orthogonal on
both $e_1+\alpha e_2$ and $e_2$. Therefore $\psi-\gamma v^1=0$. 

If $v^1\in \partial(\sigma^\vee)$ then $\beta=+\infty$, and 
if $v^1\in \Int \sigma^\vee$ then $\beta<\infty$.
We distinguish two subcases.

\underline{Subcase $v^1\in \partial(\sigma^\vee)$}. 
Here $\sigma=\R_+(e_1+\alpha e_2)+\R_+ e_2$. Since $\psi\in \Int\sigma^\vee$,
we have $\psi\ne \gamma v^1$ and therefore $\psi'> 0$. We compute $\langle \psi,e_1+\alpha e_2\rangle=\gamma$,
$\langle \psi,e_2\rangle=\gamma\psi'$. Therefore
$$
\gamma(m)=\gamma \cdot \min( 
\frac{1}{\langle m,e_1+\alpha e_2\rangle},\frac{\psi'}{\langle m,e_2\rangle}), \ m\in G^*\cap \sigma^\vee\setminus 0.
$$
We claim that $v^1$ maximizes $\gamma(\cdot)$ if and only if
$$
\psi'\le 1.
$$ 
Indeed, $v^2\in G^*\cap \sigma^\vee$ and 
$\gamma(v^2)=\gamma\min(\frac{1}{\alpha},\psi')$ is at most $\gamma(v^1)=\gamma$
if and only if $\psi'\le 1$. Conversely, suppose $\psi'\le 1$. Let $m\in G^*\cap \sigma^\vee\setminus 0$.
If $\langle m,e_2\rangle=0$, then $m=qv^1$ for some $q\in \Z_{\ge 1}$ and 
$\gamma(m)=\frac{\gamma(v^1)}{q}\le \gamma(v^1)$. If $\langle m,e_2\rangle>0$, then 
$\langle m,e_2\rangle\ge 1$ and $\frac{\psi'}{\langle m,e_2\rangle} \le 1$.

If $\psi'< 1$, $v^1$ is the unique maximizer. If $\psi'=1$ and $\alpha\ne 0$, the maximizers are $v^1,v^2$.
If $\psi'=1$ and $\alpha=0$, the maximizers are $v^1,v^2,v^1+v^2$.

We have $G\cap \Int \sigma =\{p_1e_1+p_2e_2\vert p_1,p_2\in \Z, p_1>0,
p_2>p_1\alpha\}$. We compute 
$$
\langle \psi-\gamma v^1,p_1e_2+p_2e_2\rangle=\gamma \psi' (p_2-p_1\alpha),
\
\langle \psi,p_1e_2+p_2e_2\rangle=
\gamma p_1+\gamma \psi' (p_2-p_1\alpha).
$$ 
Therefore
$
\lambda=\gamma\inf\{p_1+\psi'(1-\{p_1\alpha\}) \vert p_1\ge 1\}.
$
Since $0< \psi'\le 1$, we obtain
$$
\lambda=\gamma+\gamma \psi'(1-\alpha).
$$
Moreover, $\lambda$ is attained only by $e_1+e_2$.

Recall $\gamma(v^2)=\gamma \psi'$. Denote $m^1=v^1,m^2=v^2$. Then
$$
\psi=\gamma(1-\psi'\alpha)m^1+\gamma\psi' m^2.
$$

Finally, we compute
$
\lambda'=\inf\{\langle \psi-\gamma v^1,e\rangle \vert e\in G\cap \Int \sigma \},
$
hence
$$
\lambda'=\gamma\psi'\inf\{1-\{p_1\alpha\} \vert p_1\ge 1\}.
$$
If $\alpha\notin \Q$, then $\lambda'=0$. If $\alpha\in \Q$, then
$\lambda'=\frac{1}{q}\gamma\psi'$, where $q\ge 1$ is minimal with $q\alpha\in \Z$.

\underline{Subcase $v^1\in \Int \sigma^\vee$}. 
Here $\sigma=\R_+(e_1+\alpha e_2)+\R_+ (e_1+\beta e_2)$.
We compute $\langle \psi,e_1+\alpha e_2\rangle=\gamma$ and $\langle \psi,e_1+\beta e_2\rangle=\gamma+\gamma\psi'(\beta-\alpha)$.
Therefore
$$
\gamma(m)=\gamma\min(\frac{1}{\langle m,e_1+\alpha e_2\rangle},
\frac{1+\psi'(\beta-\alpha)}{\langle m,e_1+\beta e_2\rangle}), \ m\in G^*\cap \sigma^\vee\setminus 0.
$$
Denote $c=1+\psi'(\beta-\alpha)$.
The property that $v^1$ maximizes $\gamma(\cdot)$ is equivalent to the inequalities
$$
\max(\langle m,e_1+\alpha e_2\rangle,
\frac{\langle m,e_1+\beta e_2\rangle}{c})\ge 1 \ \forall m\in G^*\cap \sigma^\vee\setminus 0.
$$
If $\langle m,e_1+\alpha e_2\rangle\ge 1$, the inequality is satisfied. Therefore suffices to impose the inequalities in the case $\langle m,e_1+\alpha e_2\rangle< 1$.
But $\langle m,e_1+\alpha e_2\rangle\ge 0$. This is equivalent to $m=-\lfloor z_2\alpha\rfloor v^1+z_2v^2$.
Note that such $m$ belongs to $G^*\cap \sigma^\vee\setminus 0$ if and only if $z_2\in \Z\setminus 0$
and $z_2\beta\ge \lfloor z_2\alpha\rfloor$. With the notation $n=|z_2|$,
we conclude that $v^1$ maximizes $\gamma(\cdot)$ if and only if 
\begin{itemize}
	\item[a)] $n\beta-\lfloor n\alpha\rfloor\ge c$ for all $n\in \Z_{\ge 1}$, and 
	\item[b)] If $n\in \Z_{\ge 1}$ and $\lceil n\alpha\rceil\ge n\beta$, then $\lceil n\alpha\rceil- n\beta\ge c$.
\end{itemize}
 
From a) with $n=1$ we obtain (same as imposing $\gamma(v^2)\le \gamma(v^1)$)
$$
\beta\ge 1+\psi'(\beta-\alpha).
$$
In particular, 
$\psi'\le \frac{\beta-1}{\beta-\alpha}<1$. Therefore 
$$
0\le \psi'<1\le \beta.
$$
For $p_1,p_2\in \Z$,
an element $p_1e_1+p_2e_2$ belongs to $G\cap \Int \sigma$ if and only 
if $p_1\ge 1$ and $\alpha p_1<p_2<\beta p_1$. Therefore we obtain
$$
\lambda=\gamma \inf\{p_1+\psi'(p_2-\alpha p_1) \vert p_1>0,\alpha p_1<p_2<\beta p_1\}.
$$

Since $\beta\ge 1$, we have two possibilities:

(A) $\beta=1$. It follows that $\psi'=0$. From b) with $n=1$, we obtain $\alpha=0$. Here $\sigma$ is spanned by $\{e_1,e_1+e_2\}$, a basis of $G$.
The dual basis $\{m^1,m^2\}$ of $G^*$ spans $\sigma^\vee$,  
$v^1=m^1+m^2$, $\psi=\gamma v^1=\gamma m^1+\gamma m^2$, and
$
\lambda=2\gamma.
$

(B) $\beta>1$. Here $e_1+e_2\in G\cap \Int \sigma$. We obtain
$$
\lambda=\gamma(1+\psi'(1-\alpha)).
$$ 
Moreover, $\lambda$ is attained only by $e_1+e_2$ if $\psi'>0$.

We have $v^2\in G^*\cap \sigma^\vee$ and
$\gamma(v^2)=\gamma\frac{1+\psi'(\beta-\alpha)}{\beta}$.
Denote $m^1=v^1,m^2=v^2$. Then
$$
\psi=\gamma(1-\psi'\alpha)m^1+\gamma\psi'm^2.
$$
Finally, we compute 
$\lambda'=\inf\{\langle \psi-\gamma v^1,e\rangle \vert e\in G\cap \Int \sigma\}$.
A direct computation gives
$$
\lambda'=\gamma\psi'\inf\{p_2-\alpha p_1 \vert p_1\ge 1,\alpha p_1<p_2<\beta p_1\}.
$$
In case (A), $\psi'=0$ so $\lambda'=0$. In case (B), we obtain
$$
\lambda'=\gamma\psi'\inf\{1-\{p_1\alpha\} \vert p_1\ge 1\}.
$$
If $\alpha\notin \Q$, then $\lambda'=0$. Otherwise, let $q\ge 1$ minimal
such that $q\alpha\in \Z$. Then $\lambda'=\frac{1}{q}\gamma\psi'$.


\subsection{Conclusion in case (III)}

We constructed a basis $m^1,m^2$ of the two dimensional lattice $G^*$ such that $m^1,m^2\in \sigma^\vee$ and 
$$
\psi=(\lambda-\frac{\lambda-\gamma}{1-\alpha})m^1+ 
\frac{\lambda-\gamma}{1-\alpha}m^2.
$$

\begin{proof}
In case (A), $\alpha=0$, $\lambda=2\gamma$ and $\psi=\gamma m^1+\gamma m^2$, so the claim holds. In case (B), we have
$
\lambda-\gamma=\gamma \psi'(1-\alpha),
$
and the claim holds.
\end{proof}

We have $\psi'=0$ in case (A) where $\lambda=2\gamma$, and in case (B) when $\lambda=\gamma$. If $\psi'>0$, then $\lambda>\gamma$.

\begin{lem}\label{dif}
Let $t\le \lambda$. Then 
$\psi-[(t-\frac{t-\gamma}{1-\alpha})m^1+ 
\frac{t-\gamma}{1-\alpha}m^2]\in \partial(\sigma^\vee)$.
\end{lem}

\begin{proof} Denote $\varphi=-\alpha m^1+m^2$. We claim that 
$\varphi\in \partial(\sigma^\vee)$. Indeed, $\langle \varphi,e_1+\alpha e_2\rangle=0$ and $\langle \varphi,e_1+\beta e_2\rangle$ equals $\beta-\alpha$ if $\beta<+\infty$, and $1$ if $\beta=+\infty$.
Finally, we compute
$$
\psi-[(t-\frac{t-\gamma}{1-\alpha})m^1+ 
\frac{t-\gamma}{1-\alpha}m^2]=
\frac{\lambda-t}{1-\alpha}\varphi.
$$

\end{proof}

\begin{prop}
$\gamma\le \lambda\le 2\gamma$.
\end{prop}

\begin{prop}
If $\lambda>\gamma$, then $G\simeq \Z^2$ and $\lambda$ is attained
by a unique element.
\end{prop}

 Suppose $0<\alpha\in \Q$. Let $q\alpha\in \Z$ with $q\ge 1$
minimal. Then 
$
\lambda=\gamma+(q-q\alpha)\lambda'.
$


\section{Applications}


\begin{thm}\label{2se}
	Let $V$ be a two dimensional $\R$-vectors space, let $\sigma\subset V$ be a closed, strongly convex $2$-dimensional cone, with dual cone $\sigma^\vee\subset V^*$. Let 
$\psi\in \sigma^\vee\setminus 0$.
Let $G\le (V,+)$ be a subgroup, with dual subgroup $G^*\le (V^*,+)$. Suppose $G\cap \Int \sigma \ne \emptyset$.
Let $t>0$. The following properties are equivalent:
\begin{itemize}
\item[(1)] $G\cap \{e\in \Int \sigma \vert \langle \psi,e\rangle<t\}=\emptyset$.
\item[(2)] One of the following holds:
\begin{itemize}
\item[a)] There exists $m\in G^*\cap \sigma^\vee\setminus 0$ such
that $\psi-tm\in \sigma^\vee$, or
\item[b)] $G=(m^1)^*\cap(m^2)^*$, where $m^1,m^2\in \sigma^\vee$ are
linearly independent, and $\psi=t_1m^1+t_2m^2$ for some
$t_1,t_2>0$ and $t_1+t_2\ge t$.
\end{itemize}
\end{itemize}
\end{thm}

\begin{proof}
$(2)\Longrightarrow(1)$: Let $e\in G\cap \Int \sigma $.
In case a), $\langle m,e\rangle=k$ is a positive integer. Then
$$
\langle \psi,e\rangle =\langle \psi-tm,e\rangle+
t\langle m,e\rangle\ge tk\ge t. 
$$
In case b), $\langle m^i,e\rangle=k_i$ are positive integers.
Then 
$$
\langle \psi,e\rangle =t_1\langle m^1,e\rangle+t_2\langle m^2,e\rangle
=t_1k_1+t_2k_2\ge t_1+t_2\ge t. 
$$

$(1)\Longrightarrow(2)$: Let
$\lambda=\inf\{\langle \psi,e\rangle \vert e\in G\cap \Int \sigma \}$.
By assumption, $\lambda\ge t>0$. We use the notations and results
of the previous section. Let $\gamma=\gamma(v^1)$. If 
$\gamma\ge t$, a) holds with $m=v^1$. Suppose $\gamma<t$.
In particular, $\gamma<\lambda$. Then we are in case (III),
and we can take 
$t_1=\lambda-\frac{\lambda-\gamma}{1-\alpha}$,
$t_2=\frac{\lambda-\gamma}{1-\alpha}$.
Note that $t_1,t_2>0$ since $\gamma<\lambda$.
\end{proof}

\begin{rem}
The assumption $G\cap \Int \sigma \ne \emptyset$ is redundant if
$\psi\in \Int \sigma^\vee$. Indeed, by Lemma~\ref{co},
$G\cap \Int \sigma=\emptyset$ is equivalent to the existence
of $m_0\in \sigma^\vee\setminus 0$ with $\R m_0\subseteq G^*$.
If $\psi\in \Int \sigma^\vee$, then there exists $k\gg 0$
with $k\psi-m_0\in \sigma^\vee$. Then (2).a) holds for $m=\frac{m_0}{kt}$.
\end{rem}

The following is an effective version of a result of Lawrence~\cite{Law91}.

\begin{thm}\label{Lawr}
Let $\Z^2\le G\le \R^2$ be a subgroup.
Let $p,q$ be positive integers and denote
$U=\{(x,y)\in \R^2 \vert x,y>0,x+y<\frac{p}{q}\}$. 
The following properties are equivalent:
\begin{itemize}
\item[(1)] $G\cap U=\emptyset$.
\item[(2)] One of the following holds:
\begin{itemize}
\item[a)] $G\subseteq m^*$ for some 
$m\in \check{\N}^2\cap [0,\frac{q}{p}]^2\setminus 0$, or
\item[b)] $G=(m^1)^*\cap(m^2)^*$, where $m^1,m^2\in \check{\N}^2$ form a 
basis of $\check{\R}^2$, and there are integers $k_1,k_2\ge 1$
with $k_1+k_2\le 2q$ and $\frac{k_1m^1+k_2m^2}{k_1+k_2}\in [0,\frac{q}{p}]^2$.
\end{itemize}
\end{itemize}
\end{thm}

\begin{proof}
The implication $(2)\Longrightarrow (1)$ is clear.
Consider the converse. Here $\sigma=\R_+(1,0)+\R_+(0,1)$, 
$\psi=(1,1)$ and $(0,1),(1,0)\in G$.
We have $\lambda\ge \frac{p}{q}$.
If $\gamma\ge \frac{p}{q}$, we are in case a). Suppose 
$\gamma<\frac{p}{q}$. In particular, $\gamma<\lambda$.
Therefore we are in case (III).
The vector $e_1+\alpha e_2$ belongs
to a ray of $\sigma$, say $\R_+(1,0)$. 
It follows that 
$$
(1,0)=\langle v^1,(1,0)\rangle \cdot (e_1+\alpha e_2)
$$
and 
$$
\gamma=\frac{\langle \psi,(1,0)\rangle}{\langle v^1,(1,0)\rangle}=\frac{1}{\langle v^1,(1,0)\rangle}
$$
Then $\gamma=\frac{1}{l}$ for an integer $l$, and $l\alpha\in \Z$. 
The assumption $\gamma<\frac{p}{q}$ becomes $lp>q$. By Lemma~\ref{dif},
we have 
$$
\psi\in (\frac{p}{q}-\frac{\frac{p}{q}-\gamma}{1-\alpha})m^1+
\frac{\frac{p}{q}-\gamma}{1-\alpha} m^2+\sigma^\vee.
$$
We obtain
$$
(l-l\alpha)q\psi\in (q-pl\alpha)m^1+(lp-q)m^2+\sigma^\vee.
$$
Take $k_1=q-pl\alpha,k_2=lp-q$. Then $k_1,k_2>0$ and 
$k_1+k_2=p(l-l\alpha)$. Now $l-l\alpha\le l=\frac{1}{\gamma}
\le \frac{2}{\lambda}\le \frac{2q}{p}$. Therefore
$k_1+k_2\le 2q$. 
\end{proof}

\begin{rem}
If $\frac{p}{q}=\frac{1}{k}$ and $k>1$, 
we may take $k_1+k_2<2k$ in b). Indeed,
suppose $l-l\alpha=2k$. Then 
$\alpha=0,\gamma=\frac{\lambda}{2},\lambda=\frac{1}{k}$, and
$\psi=\frac{1}{k}[\frac{1}{2}m^1+\frac{1}{2}m^2]$. So we
may take $k_1=k_2=1$ in the case.
\end{rem}

\begin{exmp} If $\frac{p}{q}=1$, we obtain that
\begin{itemize}
\item[a)] either $G$ is contained in one of the subgroups 
$(1,0)^*,(0,1)^*,(1,1)^*$, 
\item[b)] $G$ equals $(0,1)^*\cap (1,0)^*=\Z^2$. 
\end{itemize}
If $\frac{p}{q}=\frac{1}{2}$, up to interchanging the coordinates, we obtain 
\begin{itemize}
\item[a)] either $G$ is contained in one of the subgroups 
$(0,2)^*$, $(1,2)^*$, $(2,2)^*$, 
\item[b)] or $G$ equals one of the following subgroups:
$(0,3)^*\cap (3,0)^*, (0,3)^*\cap (3,1)^*, 
(0,3)^*\cap (4,0)^*, (0,3)^*\cap (4,1)^*, (0,3)^*\cap (5,0)^*,
(0,4)^*\cap (4,0)^*, (1,3)^*\cap(3,1)^*,(1,3)^*\cap(4,0)^*.
$
\end{itemize}
\end{exmp}


\section{Two dimensional toric log germs}


Let $X$ be an affine toric surface with invariant point $P$.
Let $\Sigma=X\setminus T$ be the complement of the torus. Then
$\Sigma=E_1+E_2$ is the sum of two invariant prime divisors $E_1,E_2$. 
Since $\dim X=2$, both $E_i$ are $\Q$-Cartier. 
We have 
$$
K+\Sigma= 0,\mld_P(X,\Sigma)=0.
$$
Up to isomorphism, $X=T_N\emb(\sigma)$, where $\sigma=\R^2_{\ge 0}$ is the standard positive cone in $\R^2$, and $N\subset \R^2$ is a $2$-dimensional lattice which contains $e_1=(1,0),e_2=(0,1)$ as primitive vectors.

Let $P\in (X,B)$ be an invariant  structure of log variety with 
$\mld_P(X,B)=a\ge 0$. We have $B=b_1E_1+b_2E_2$, with $b_1,b_2\in [0,1]$.
Let $\psi\in N^*=:M$ with $\langle \psi,e_i\rangle=1-b_i$ for $i=1,2$,
so that $(\chi^\psi)+K+B=0$.
For any toric valuation $E_e$ of $X$ (here $e\in N^{prim}\cap \sigma$),
the log discrepancy of $(X,B)$ in $E_e$ is computed as follows:
$$
a_{E_e}(X,B)=\langle \psi,e\rangle.
$$
In particular, $a=\min\{\langle \psi,e\rangle \vert e\in N\cap \Int\sigma\}$.
We have $\mld_P(X,B)\ge t$ if and only if $N\cap \{e\in \Int \sigma \vert \langle \psi,e\rangle<t\}=\emptyset$.

{\bf 1)} $a\in [0,2]$. Note that $a=0$ if and only if $B=\Sigma$, 
and $a=2$ if and only if $X=\bA^2$ and $B=0$. By Theorem~\ref{2se},
we deduce that $a\ge 1$ if and only if $(X,B)$ is one of the following:
\begin{itemize}
\item $(\bA^1,b_1\cdot 0)\times (\bA^1,b_2\cdot 0)$,
$b_1+b_2\le 1$ (here $a=2-b_1-b_2$), or
\item $B=0$ and $X$ is a quotient of $\bA^2$ of type $\frac{1}{l+1}(l,1) \ (l\ge 1)$. Here $a=1$.
\end{itemize}

{\bf 2)} Denote by $|\fm_{X,P}|_\Sigma$ the invariant hyperplane sections through $P$, or equivalently, the hyperplane sections through $P$ which are supported by $\Sigma$. One of the following holds:
\begin{itemize}
\item[a)] There exists $H\in |\fm_{X,P}|_\Sigma$, $\mld_P(X,B+aH)=0$, or
\item[b)] There exist $H_1,H_2\in |\fm_{X,P}|_\Sigma$ and $\gamma_1,\gamma_2>0$ such that $\gamma_1+\gamma_2=a$ and $\mld_P(X,B+\gamma_1H_1+\gamma_2H_2)=0$.
\end{itemize}

\begin{proof}
Let $\gamma=\max\{\lct_P(X,B;H) \vert H\in |\fm_{X,P}|_\Sigma\}$. Choose
$H_1\in |\fm_{X,P}|_\Sigma$ with $(X,B+\gamma H_1)$ maximally log canonical.
We have $a\ge \gamma$. If $a=\gamma$, a) holds with $H=H_1$.

Suppose $a>\gamma$. Note that $(\psi)=\Sigma-B$. 
We constructed $\alpha\in [0,1)$ and 
$H_2\in |\fm_{X,P}|_\Sigma$  
with $B+\gamma_1H_1+\gamma_2H_2=\Sigma$, where 
$$
\gamma_2=\frac{a-\gamma}{1-\alpha}, \gamma_1=\lambda-\gamma_2=
\gamma-\alpha \gamma_2\le \gamma.
$$
We have a chain which strictly decreases the minimal log canonical
center near $P$
$$
(X,B)\prec (X,B+\gamma H_1)\prec (X,B+\gamma_1 H_1+\gamma_2H_2)=(X,\Sigma).
$$
Here $B+\gamma_1 H_1+\gamma_2H_2=(B+\gamma H_1)+\gamma_2(H_2-\alpha H_1)$.
Note that $H_2-\alpha H_1$ is effective $\Q$-Cartier. Its index coincides
with the index of $\alpha$. If $\mult_{E_1}(B+\gamma H_1)=0$ and 
$l=\mult_{E_1}(H_1)$, then $\ind(\alpha)=l$ and $\gamma=\frac{1-b_1}{l}$.
In particular, $l\gamma\le 1$. We showed that 
$$
\frac{a}{2}\le \gamma\le a,
$$
so $l\le \frac{2}{a}$. Therefore $l$ is bounded by $a$.
Also, the inequalities $\psi'\le 1$ and $\alpha<1$ become
$$
a-\gamma\le \gamma_2\le \gamma.
$$
Let $a_1=\mld_P(X,B+\gamma H_1)$. We showed that
\begin{itemize}
\item $a=2\gamma$ if $\alpha=0$,
\item $a=\gamma+(l-l\alpha)a_1$ if $\alpha\ne 0$.
\end{itemize}

Suppose $b_1,b_2$ belong to a given DCC set $\cB\subset [0,1]$.
Then $a$ satisfies ACC. In fact, note that $\gamma<a$ and $b_1\in\cB$ implies
$a-\gamma\ge \epsilon(a,l,\cB)>0$. One shows that if $a-\gamma$ is bounded away from $0$, 
then $X$ belongs to finitely many isomorphism types. So
$a$ can only accumulate to $\gamma$, from above.
\end{proof}

\begin{cor}
There exists $H\in |\fm_{X,P}|_\Sigma$ with $\mld_P(X,B+\frac{a}{2}H)\ge 0$.
\end{cor}

\begin{proof}
Suppose in b) that $\gamma_1\ge \gamma_2$. Then $\gamma_1\ge \frac{a}{2}$,
so we may take $H=H_1$.
\end{proof}

{\bf 3)} Suppose $B$ has standard coefficients and 
$\mld_P(X,B)\ge \frac{p}{q}$. Then there exists $1\le s\le \frac{2q}{p}$
and $B\le B_{qs}\le \Sigma$ such that $qs(K+B_{qs})\sim 0$ and 
$\mld_P(X,B_{qs})\ge \frac{p}{q}$.

\begin{proof} Case a): there exists $0\ne m\in M\cap \sigma^\vee$
with $\psi-\frac{p}{q}m\in \sigma^\vee$. Set 
$B_q=\sum_{i=1}^2(1-\frac{pm_i}{q})E_i$. Then $B\le B_q\le \Sigma$
and $(\chi^{pm})+q(K+B_q)=0$. Moreover, $\psi_{K+B_q}=\frac{p}{q}m$, so
$$
\langle \frac{p}{q}m,e\rangle\ge \frac{p}{q},\forall e\in N\cap \Int\sigma.
$$
Case b): $\psi=(a-\frac{a-\gamma}{1-\alpha})m^1+\frac{a-\gamma}{1-\alpha}m^2$.
In particular, 
$$
\psi\ge (\frac{p}{q}-\frac{\frac{p}{q}-\gamma}{1-\alpha})m^1+
\frac{\frac{p}{q}-\gamma}{1-\alpha}m^2.
$$
Equivalently,
$$
q\psi\ge (p-\frac{p-q\gamma}{1-\alpha})m^1+\frac{p-q\gamma}{1-\alpha}m^2.
$$
Since $B$ has standard coefficients, $\gamma=\frac{1}{l}$ for an
integer $l$ with $l\alpha\in \Z$. The above inequality becomes 
$$
(l-l\alpha) q\psi\ge (q-pl\alpha)m^1+(lp-q)m^2.
$$
Let $s=l-l\alpha$. Then $sq\psi\ge z_1m^1+z_2m^2$ for some
$z_1,z_2\in \Z_{>0}$. Set $B_{sq}=\sum_{i=1}^2(1-\frac{z_1m^1_i+z_2m^2_i}{sq})E_i$.
Then $B\le B_{sq}\le \Sigma$ and $(\chi^{z_1m^1+z_2m^2})+sq(K+B_{sq})=0$.
Thus 
$$
\langle \psi_{K+B_{sq}},e\rangle=
\langle \frac{z_1m^1+z_2m^2}{sq},e\rangle\ge \frac{z_1+z_2}{sq}=\frac{p}{q}.
$$
Note $s=l-l\alpha\le l=\gamma^{-1}\le \frac{2}{a}$.
\end{proof}

\begin{cor}
Suppose $B$ has standard coefficients and $\mld_P(X,B)=a>0$. 
Let $q$ be a positive integer with $qa\in \Z$. Then $sq(K+B)\sim 0$
for some integer $1\le s\le \frac{2}{a}$.
\end{cor}

{\bf 4)} The following hold:
\begin{itemize}
\item There exists $1\le n\le \lceil \frac{2}{a}\rceil$,
$B\le B_n\le \Sigma$ with $n(K+B_n)\sim 0$, $\mld_P(X,B_n)>0$.
\item There exists $1\le n\le \frac{2}{a}$,
$\lfloor B\rfloor+\frac{\lfloor (n+1)\{B\}\rfloor}{n}\le 
B_n\le \Sigma$ with $n(K+B_n)\sim 0$, $\mld_P(X,B_n)>0$.
\end{itemize}

\begin{proof}
Use $H_1$ and $\gamma\le \frac{2}{a}$.
\end{proof}


\end{document}